\documentclass[reqno,a4paper, 11pt]{amsart}

\usepackage[a4paper=true,pdfpagelabels]{hyperref}
\usepackage{graphicx}

\usepackage[ansinew]{inputenc}
\usepackage{amsfonts,epsfig}
\usepackage{latexsym}
\usepackage{amsmath}
\usepackage{amssymb}
\usepackage{mathabx}

\newtheorem{theorem}{Theorem}
\newtheorem{lemma}[theorem]{Lemma}
\newtheorem{corollary}[theorem]{Corollary}
\newtheorem{proposition}[theorem]{Proposition}

\newtheorem{lettertheorem}{Theorem}

\theoremstyle{definition}

\theoremstyle{remark}

\numberwithin{equation}{section}

\newcommand{\D}{\mathbb{D}}
\newcommand{\DD}{\widehat{\mathcal{D}}}

\newcommand{\Dd}{\widecheck{\mathcal{D}}}
\newcommand{\DDD}{\mathcal{D}}

\newcommand{\N}{\mathbb{N}}

\newcommand{\R}{\mathcal{R}}

\newcommand{\e}{\varepsilon}
\newcommand{\ep}{\varepsilon}

\renewcommand{\phi}{\varphi}

\DeclareMathOperator*{\esssup}{ess\,sup}

\def\a{\alpha}       \def\b{\beta}        
\def\d{\delta}           \def\e{\varepsilon}
\def\la{\lambda}     \def\om{\omega}      
       \def\t{\theta}       
         \def\r{\rho}         \def\z{\zeta}

\renewcommand{\H}{\mathcal{H}}

\newenvironment{Prf}{\noindent{\emph{Proof of}}}
{\hfill$\Box$ }
\begin{document}

\title[Radial averaging operator acting on Bergman and Lebesgue spaces]{Radial averaging operator acting on Bergman and Lebesgue spaces}


\thanks{This research was supported in part by Ministerio de Econom\'{\i}a y Competitivivad, Spain, projects
MTM2014-52865-P and MTM2017-90584-REDT; La Junta de Andaluc{\'i}a,
project FQM210; Academy of Finland project no. 268009, and by Faculty of Science and Forestry of University of Eastern Finland.}

\date{\today}

\begin{abstract}
It is shown that the radial averaging operator
	$$
  T_\om(f)(z)=\frac{\int_{|z|}^1f\left(s\frac{z}{|z|}\right)\om(s)\,ds}{\widehat{\om}(z)},\quad \widehat{\om}(z)=\int_{|z|}^1\om(s)\,ds,
  $$
induced by a radial weight $\om$ on the unit disc $\D$, is bounded from the weighted Bergman space~$A^p_\nu$, where $0<p<\infty$ and the radial weight $\nu$ satisfies $\widehat{\nu}(r)\le C\widehat{\nu}\left(\frac{1+r}{2}\right)$ for all $0\le r<1$, to $L^p_\nu$ if and only if the self-improving condition $\sup_{0\le r<1}\frac{\widehat{\om}(r)^p}{\int_{r}^1 s\nu(s)\,ds}\int_0^r\frac{t\nu(t)}{\widehat{\om}(t)^p}\,dt<\infty$ is satisfied. Further, two characterizations of the weak type inequality
    \begin{equation*}
    \eta\left(\left\{z\in\D:|T_\om(f)(z)|\ge\lambda\right\}\right)\lesssim\lambda^{-p} \| f\|_{L^p_\nu}^p,\quad \lambda>0,
    \end{equation*}
are established for arbitrary radial weights $\om$, $\nu$ and $\eta$. Moreover, differences and interrelationships between the cases $A^p_\nu\to L^p_\nu$, $L^p_\nu\to L^p_\nu$ and $L^p_\nu\to L^{p,\infty}_\nu$ are analyzed.
\end{abstract}

\author{Taneli Korhonen}
\address{University of Eastern Finland, P.O.Box 111, 80101 Joensuu, Finland}
\email{taneli.korhonen@uef.fi}

\author{Jos\'e \'Angel Pel\'aez}
\address{Departamento de An\'alisis Matem\'atico, Universidad de M\'alaga, Campus de Teatinos,   29071 M\'alaga, Spain}
\email{japelaez@uma.es}

\author{Jouni R\"atty\"a}
\address{University of Eastern Finland, P.O.Box 111, 80101 Joensuu, Finland}
\email{jouni.rattya@uef.fi}

\maketitle

\section{Introduction and main results}

A function $\om:\D\to[0,\infty)$, integrable over the unit disc $\D$, is called a weight. It  is radial if $\om(z) = \om(|z|)$ for all $z\in\D$. For $0<p<\infty$ and a weight $\om$, the Lebesgue space $L^p_\om$ consists of complex-valued measurable functions $f$ in $\D$ such that
    $$
    \|f\|_{L^p_\om}=\left(\int_\D|f(z)|^p\om(z)\,dA(z)\right)^{\frac1p}<\infty,
    $$
where $dA(z) = \frac{dx\,dy}\pi$ denotes the element of the normalized Lebesgue area measure on $\D$. The weighted Bergman space $A^p_\om$ is the space of analytic functions in $L^p_\om$.

For a radial weight $\om$, we assume throughout the paper that $\widehat{\om}(z) = \int_{|z|}^1 \om(s)\,ds > 0$ for all $z \in \D$, otherwise the Bergman space $A^p_\om$ would contain all analytic functions in $\D$.
For a radial weight, the norm convergence in the Hilbert space $A^2_\om$ implies the uniform convergence on compact subsets of $\D$ and so the point evaluations are bounded linear functionals on $A^2_\om$. Therefore there exist reproducing Bergman kernels $B^\om_z\in A^2_\om$ such that
    $$
    f(z)=\langle f,B^\om_z \rangle_{A^2_\om}
    = \int_\D f(\zeta)\overline{B^\om_z(\zeta)}\om(\zeta)\,dA(\zeta),\quad z\in\D, \quad f\in A^2_\om.
    $$
The Hilbert space $A^2_\om$ is a closed subspace of $L^2_\om$, and hence the orthogonal Bergman projection from $L^2_\om$ to $A^2_\om$ is given by
    $$
    P_\om(f)(z) = \int_\D f(\zeta)\overline{B^\om_z(\zeta)}\om(\zeta)\,dA(\zeta),\quad z\in\D.
    $$

A radial weight $\om$ belongs to the class~$\DD$ if there exists a constant $C=C(\om)>1$ such that the doubling condition $\widehat{\om}(r)\le C\widehat{\om}(\frac{1+r}{2})$ is satisfied for all $0\le r<1$. Moreover, if there exist $K=K(\om)>1$ and $C=C(\om)>1$ such that
    \begin{equation}\label{K}
    \widehat{\om}(r)\ge C\widehat{\om}\left(1-\frac{1-r}{K}\right),\quad 0\le r<1,
    \end{equation}
then we write $\om\in\Dd$. The intersection $\DD\cap\Dd$ is denoted by $\DDD$.
The definitions of these classes of weights are of geometric nature, and the classes themselves arise naturally in the study of
classical operators. For example, it is known that the Bergman projection $P_\om$ induced by a radial weight $\om$ is bounded from $L^\infty$ of $\D$ to the Bloch space $\mathcal{B}$ if and only if $\om\in\DD$, and further, $P_\om:L^\infty\to\mathcal{B}$ is bounded and onto if and only if $\om\in\DDD$~\cite{PR2018}. A radial weight $\om$ is regular if $\widehat{\om}(r)\asymp\om(r)(1-r)$ for all $0\le r<1$. The class of regular weights is denoted by $\R$,
and $\R\subsetneq\DDD$. For basic properties of these classes of weights and more, see~\cite{PelSum14,PR2014Memoirs,PR2018} and the references therein.

This paper concerns the radial (Hardy) averaging operator $T_\om$, induced by a radial weight~$\om$, and defined by
  $$
  T_\om(f)(z)=\frac{\int_{|z|}^1f\left(s\frac{z}{|z|}\right)\om(s)\,ds}{\widehat{\om}(z)},\quad z\in\D\setminus\{0\},
  $$
and its maximal version
	$$
  T^N_\om(f)(z)=\frac{\int_{|z|}^1N(f)\left(s\frac{z}{|z|}\right)\om(s)\,ds}{\widehat{\om}(z)},\quad z\in\D\setminus\{0\},
  $$
where $N(f)(z)=\sup_{\zeta\in\Gamma(z)}|f(\zeta)|$ is the maximal function on the
non-tangential approach region
		$$
		\Gamma(z)=\left\{\xi\in\D:|\theta-\arg(\xi)|<\frac{1}{2}\left(1-\frac{|\xi|}{r}\right )\right\},
		\quad z=re^{i\theta}\in\overline{\D}\setminus\{0\},
    $$
with vertex in the punctured closed unit disc.

Hardy averaging operators have been extensively studied since decades ago in the context of Lebesgue spaces, see
\cite{AndersenMuckStudia82,KPB,DMROIndiana12,GoLa,Muckenhoupt1972,SawyerTams84} and the references therein. The question of when $T_\om:L^p_\nu\to L^p_\eta$ is bounded can be answered  by a classical result due to Muckenhoupt~\cite{Muckenhoupt1972}. Namely, the condition
	$$
    M_p
		=M_p(\om,\nu,\eta)
		=\sup_{0<r<1}\left(\int_r^1\left(\frac{\om(s)}{s\nu(s)}\right)^{p'}s\nu(s)\,ds\right)^\frac1{p'}
    \left(\int_0^r\frac{\eta(s)s}{\widehat{\om}(s)^p}\,ds\right)^\frac1p<\infty
    $$
being a sufficient condition for $T_\om:L^p_\nu\to L^p_\eta$ to be bounded and an estimate from above for the operator norm follow by applying \cite[Theorem~2]{Muckenhoupt1972}, and a testing similar to that in the proof of the said theorem gives the necessity and a lower bound for the operator norm. We state this result for further reference as follows.

\begin{lettertheorem}\label{Theorem:T_w-bounded-L^p-case}
Let $1<p<\infty$ and $\om$, $\nu$ and $\eta$ radial weights. Then $T_\om:L^p_\nu\to L^p_\eta$ is bounded if and only if $M_p(\om,\nu,\eta)<\infty$. Moreover, $M_p\le\|T_\om\|_{L^p_\nu\to L^p_\eta}\le p{(p-1)^\frac{1-p}{p}}M_p$.
\end{lettertheorem}

It is worth noticing that under the hypotheses $\om,\nu\in\R$ and $\eta\in\Dd$ the averaging operator~$T_\om$ and
the maximal Bergman projection
		$$
    P^+_\om(f)(z) = \int_\D f(\zeta)|B^\om_z(\zeta)|\om(\zeta)\,dA(\zeta),\quad z\in\D,
    $$
are simultaneously bounded from $L^p_\nu$ to $L^p_\eta$ by \cite{PKRthree}.
Moreover, in the particular case $\om\equiv1$ with $\nu\in\R$, $T_\om:A^p_\nu\to L^p_\nu$ is bounded if and only if the Bergman projection~$P_\om$ is bounded on~$L^p_\nu$~\cite{PGR}. Therefore the averaging operator $T_\omega$, which is interesting in its own right, is closely related to
the weighted Bergman projection $P_\om$  and it might be further used as a model for its study.
With this aim and that of studying the difference respect to their action on Lebesgue spaces, we consider the averaging operator $T_\om$ acting from $ A^p_\nu$ to $L^p_\nu$ when $\nu\in\DD$.

We recall that a positive Borel measure $\mu$ on $\D$ is a $q$-Carleson measure for $A^p_\nu$ if the identity operator is bounded from $A^p_\nu$ to $L^q_\mu$. The $q$-Carleson measures for $A^p_\nu$ induced by $\nu\in\DD$ are described in \cite{PelRat2014}. By denoting
    $$
    d\mu_{p,\om,\nu}(z)=\widehat{\om}(z)^{p-1}\om(z)\left(\int_0^{|z|}\frac{\nu(s)}{\widehat{\om}(s)^p}s\,ds\right)dA(z),\quad z\in\D,
    $$
we state the first main result of the paper as follows.

\begin{theorem}\label{Theorem:T_w-bounded-intro}
Let $0<p<\infty$, $\nu\in\DD$ and $\om$ a radial weight. Then the following statements are equivalent:
\begin{itemize}
\item[\rm(i)] $T_\om^N:A^p_\nu\to L^p_\nu$ is bounded;
\item[\rm(ii)] $T_\om:A^p_\nu\to L^p_\nu$ is bounded;
\item[\rm(iii)] $\DD_p(\om,\nu)=\sup_{0\le r<1}\frac{\widehat{\om}(r)^p}{\int_{r}^1 s\nu(s)\,ds}\int_0^r\frac{t\nu(t)}{\widehat{\om}(t)^p}\,dt<\infty$;
\item[\rm(iv)] $\mu_{p,\om,\nu}$ is a $q$-Carleson measure for $A^q_\nu$ for some (equivalently for each) $0<q<\infty$.
\end{itemize}
Moreover,
    \begin{equation}\label{Eq:norm-estimates}
    \|T_\om^N\|^p_{A^p_\nu\to L^p_\nu}\asymp\|T_\om\|^p_{A^p_\nu\to L^p_\nu}\asymp\|Id\|^p_{A^p_\nu\to L^p_{\mu_{p,\om,\nu}}}\asymp\DD_p(\om,\nu).
    \end{equation}
\end{theorem}

To obtain the necessary condition $\DD_p(\om,\nu)<\infty$, we observe first that indicators cannot be used to build test functions, as it happens in Theorem~\ref{Theorem:T_w-bounded-L^p-case}. This is a significant difference between the analytic and the classical case. We overcome this obstacle by using derivatives of Bergman reproducing kernels $B^\nu_a$ as test functions. The asymptotic behavior of $\|(B^\nu_a)^{(N)}\|_{A^p_\nu}$ is well understood by the recent study~\cite{PelRatKernels} and these estimates are strongly in use in our reasoning. In addition, we need lower estimates for the $L^p_\nu$-norm of $T_\om\left(\left(B_a^\nu\right)^{(N)}\right)$. We prove these estimates by using Hardy-Littlewood inequalities and smooth polynomials related to Hadamard products. This approach actually enables us to obtain a sharp necessary condition for the boundedness in a much more general situation, see Proposition~\ref{Theorem:T_w-boundedpq3 weights} below. As an immediate consequence of these deductions we observe that $T_\om:L^p_\nu\to L^q_\nu$ fails to be bounded if $q>p$ and $\nu\in\DD$. Another pivotal dissimilarity between the analytic and the classical case consists of finding an appropriate way to obtain a sufficient condition in the former one. It turns out that Carleson measures is an adequate one as is already seen in the statement. We will also show that the condition $\DD_p(\om,\nu)<\infty$ is self-improving in the sense that if $\DD_p(\om,\nu)<\infty$, then there exists $\e=\e(p,\om,\nu)>0$ such that $\DD_{p-\e}(\om,\nu)<\infty$, and this observation will play an important role in the proof.

By using Theorems~\ref{Theorem:T_w-bounded-intro} and \ref{Theorem:T_w-bounded-L^p-case} along with~\cite[Theorem 11]{PelRatKernels}, we establish the following result.

\begin{corollary}\label{corollary:A^p-L^p-the same}
Let $1<p<\infty$ and $\om,\nu\in\R$. Then the following statements are equivalent:
    \begin{enumerate}
    \item[\rm(i)] $T_\om:A^p_\nu\to L^p_\nu$ is bounded; %if and only if
    \item[\rm(ii)] $T_\om:L^p_\nu\to L^p_\nu$ is bounded;
    \item[\rm(iii)] $P^+_\om:L^p_\nu\to L^p_\nu$ is bounded;
    \item[\rm(iv)] $P_\om:L^p_\nu\to L^p_\nu$ is bounded.
    \end{enumerate}
\end{corollary}

It is not a surprise that the statement in Corollary~\ref{corollary:A^p-L^p-the same} fails without any local regularity hypotheses on the weights because $\om$ being absolutely continuous with respect to $\nu$ is a necessary condition for $T_\om$ to be bounded on $L^p_\nu$ by Theorem~\ref{Theorem:T_w-bounded-L^p-case},
see Corollary~\ref{counterexample} below.

We also consider the $(p,p)$-weak type inequality for the radial averaging operator. Our main result in this direction is the following.

\begin{theorem}\label{Theorem:T_w-bounded-weak-L^p-case}
Let $1<p<\infty$ and $\om,\nu,\eta$ be radial weights.
Then the following statements are equivalent:
    \begin{enumerate}
    \item[\rm(i)] $T_\om:L^p_\nu\to L^{p,\infty}_\eta$ is bounded;
    \item [\rm(ii)]
        $\displaystyle
        N_{p}(\om,\nu,\eta)
				=\sup_{0\le t<r<1}\left(\frac{\int_t^r\eta(s)s\,ds}{\widehat{\om}(t)^p}\right)^\frac1p
        \left(\int_r^1\left(\frac{\om(s)}{s\nu(s)}\right)^{p'}s\nu(s)\,ds\right)^{\frac1{p'}}<\infty;
        $
    \item[\rm(iii)]
    $\displaystyle
    M_{p,\e}
		=M_{p,\e}(\om,\nu,\eta)
    =\sup_{0\leq r<1}\left(\widehat{\om}(r)^\e\int_0^r \frac{s\eta(s)}{\widehat{\om}(s)^{p+\ep}}\,ds\right)^\frac1p
    \left(\int_r^1 \left(\frac{\om(s)}{s\nu(s)}\right)^{p'}s\nu(s)\,ds\right)^{\frac1{p'}}<\infty
    $
for some (equivalently for each) $\e>0$.
    \end{enumerate}
Moreover, for each fixed $\ep>0$,
    \begin{equation}\label{eq:eqweaknorms}
    \|T_\om\|_{L^p_\nu\to L^{p,\infty}_\eta}\asymp N_{p}(\om,\nu,\eta)\asymp M_{p,\e}(\om,\nu,\eta).
    \end{equation}
\end{theorem}

A calculation shows that Theorem~A can be deduced from the corresponding strong inequality for the simple case $\om\equiv1$ by modifying the weights involved. However, this is no longer true for weak-type inequalities due to the nature of the level sets for $|T_\om|$. On the other hand, although conditions similar to Theorem~\ref{Theorem:T_w-bounded-weak-L^p-case}(ii) appear frequently in the literature \cite{KPB,GoLa},
conditions analogous to Theorem~\ref{Theorem:T_w-bounded-weak-L^p-case}(iii), which is a kind of generalization of the classical
result of Andersen and Muckenhoupt \cite[Theorem~2]{AndersenMuckStudia82}, seem to have been less explored. Therefore, we believe that Part (iii) in Theorem~\ref{Theorem:T_w-bounded-weak-L^p-case} adds theoretical and practical value to the result. As for the proof of Theorem~\ref{Theorem:T_w-bounded-weak-L^p-case}, it follows the leading idea of that of \cite[Theorem~2]{AndersenMuckStudia82} but a good number of non-trivial modifications are needed.

It is also worth noticing that for each $1<p<\infty$ there are weights $(\om,\nu,\eta)$,  such that  $T_\om:L^p_\nu\to L^{p,\infty}_\eta$ is bounded but $T_\om:L^p_\nu\to L^p_\eta$ is not. Nonetheless, as a byproduct of our results we deduce the following.

\begin{corollary}\label{corollary:weird}
Let $1<p<\infty$, $\om\in\DD$ and $\nu\in\DDD$. Then $T_\om:L^p_\nu\to L^p_\nu$ bounded if and only if $T_\om:A^p_\nu\to L^p_\nu$ bounded and $T_\om:L^p_\nu\to L^{p,\infty}_\nu$ bounded.
\end{corollary}

In Section~\ref{sec:analytic} we  prove Theorem~\ref{Theorem:T_w-bounded-intro} and Corollary~\ref{corollary:A^p-L^p-the same}. Theorem~\ref{Theorem:T_w-bounded-weak-L^p-case} is proved in Section~\ref{Theorem:T_w-bounded-weak-L^p-case}, where we also provide the details of the proof of Corollary~\ref{corollary:weird} and offer concrete examples related to the results obtained in this study.

Throughout the paper $\frac{1}{p}+\frac{1}{p'}=1$ for $1<p<\infty$. Further, the letter $C=C(\cdot)$ will denote an
absolute constant whose value depends on the parameters indicated
in the parenthesis, and may change from one occurrence to another.
We will use the notation $a\lesssim b$ if there exists a constant
$C=C(\cdot)>0$ such that $a\le Cb$, and $a\gtrsim b$ is understood
in an analogous manner. In particular, if $a\lesssim b$ and
$a\gtrsim b$, then we will write $a\asymp b$. An observant reader notices that this last notation is already in use in \eqref{Eq:norm-estimates} and \eqref{eq:eqweaknorms}.

\section{Analytic case}\label{sec:analytic}

\subsection{Test functions}

We first establish necessary conditions for $T_\om: A^p_\nu \to L^q_\eta$ to be bounded when $0<p\le q<\infty$, $\nu\in\DD$ and $\om,\eta$ are radial weights. To do this, some background material is needed.

For a radial weight $\om$, the normalized monomials $\frac{z^n}{\sqrt{2\int_0^1 r^{2n+1}\om(r)\,dr}}$, $n\in\N\cup\{0\}$, form the standard orthonormal basis of $A^2_\om$, and hence
    \begin{equation}\label{repker}
    B^\om_z(\z)=\sum_{n=0}^\infty\frac{(\z\bar{z})^n}{2\int_0^1 r^{2n+1}\om(r)\,dr},\quad z,\z\in\D.
    \end{equation}
Write
    $$
    \om_{t,x}=\int_t^1 s^x\om(s)\,ds,\quad 0\le x<\infty,\quad 0\le t<1,
    $$
and $\om_{x}=\om_{0,x}$ for short. For $f\in\H(\D)$ with Maclaurin series $f(z)=\sum_{n=0}^\infty\widehat{f}(n)z^n$ and $k\in\N\cup\{0\}$, denote $\Delta_k f(z)=\sum_{n=2^k}^{2^{k+1}-1} \widehat{f}(n)z^n.$
With these preparations we can state and prove the following mean estimate for the image of $(B^\nu_a)^{(N)}$ under $T_\om$.

\begin{lemma}\label{le:testfunctionsq>p}
Let $0<q<\infty$, $\nu\in\DD$, $\om$ a radial weight, $N\in\N$ and $a\in\D$ with $|a|\ge 1-\frac{1}{2N}$. Then
    \begin{equation}\label{eq:testfunctionsq>p}
    \widehat{\om}(t)^qM^q_q\left(t,T_\om\left((B^\nu_a)^{(N)}\right)\right)
    \gtrsim\left\{
        \begin{array}{rl}
        \frac{\widehat{\om}(a)^q}{\widehat{\nu}(a)^q(1-|a|)^{q(N+1)-1}}, & \quad 0\le t\le |a|,\\
        \frac{\widehat{\om}(t)^q}{\widehat{\nu}(a)^q(1-|a|)^{q(N+1)-1}}, & \quad |a|\le t<1.
        \end{array}\right.
    \end{equation}
\end{lemma}

\begin{proof}
First observe that
    \begin{equation}\label{eq:tf0}
    \begin{split}
    \widehat{\om}(t)T_\om\left((B^\nu_a)^{(N)}\right)(te^{i\theta})
    &=\int_t^1 (B^\nu_a)^{(N)}(se^{i\theta})\,\om(s)ds\\
    &=\int_t^1 \left(\sum_{j= N}^\infty    \frac{j(j-1)\cdots(j-N+1)(se^{i\theta})^{j-N}\overline{a}^{j}}{2\nu_{2j+1}} \right)\,\om(s)ds\\
    &=\sum_{j= N}^\infty\frac{j(j-1)\cdots(j-N+1) e^{i\theta(j-N)}\overline{a}^{j}}{2\nu_{2j+1}}\om_{t,j-N}.
    \end{split}
    \end{equation}
From now on we split the proof into four cases. Let first $0\le t\le |a|$ and $0<q\le2$. Then the classical Hardy-Littlewood inequality \cite[Theorem~6.2]{Duren1970} gives
    \begin{equation}
    \begin{split}\label{eq:tf1}
    \widehat{\om}(t)^qM^q_q\left(t,T_\om\left((B^\nu_a)^{(N)}\right)\right)
    &\gtrsim\sum_{j=N}^\infty \frac{j^{Nq+q-2}\om^q_{t,j-N}|a|^{jq}}{\nu^q_{2j+1}}.
    \end{split}
    \end{equation}
Choose $N^\star\in\N$ such that $1-\frac{1}{N^\star}\le |a|<1-\frac{1}{N^\star+1}$. Then the inequality $0\le t\le|a|$ and \cite[Lemma~2.1(iv)]{PelSum14} yield
    \begin{equation}
    \begin{split}\label{eq:tf2}
    \sum_{j=N}^{N^\star} \frac{j^{Nq+q-2}\om^q_{t,j-N}|a|^{jq}}{\nu^q_{2j+1}}
    &\gtrsim\sum_{j=N}^{N^\star} \frac{j^{Nq+q-2}\om^q_{|a|,j-N}}{\nu^q_{2j+1}}
    \gtrsim\widehat{\om}(a)^q\sum_{j=N}^{N^\star} \frac{j^{Nq+q-2}}{\widehat{\nu}\left(1-\frac{1}{2j+1}\right)^q}\\
    & \gtrsim\widehat{\om}(a)^q\int_{N}^{N^\star+1}\frac{s^{q(N+1)-2}}{\widehat{\nu}(1-\frac{1}{2s+1})^{q}}\, ds
    \asymp\widehat{\om}(a)^q\int_{N}^{N^\star+1}\frac{s^{q(N+1)-2}}{\widehat{\nu}(1-\frac{1}{s})^{q}}\,ds\\
    &\ge\widehat{\om}(a)^q\int_{1-\frac{1}{N}}^{1-\frac1{N^\star +1}}\frac{dt}{\widehat{\nu}(t)^{q}(1-t)^{q(N+1)}}\\
    &\ge\widehat{\om}(a)^q \int_{1-\frac{1}{N}}^{|a|}\frac{dt}{\widehat{\nu}(t)^{q}(1-t)^{q(N+1)}}\\
    &\ge\widehat{\om}(a)^q \int_{2|a|-1}^{|a|}\frac{dt}{\widehat{\nu}(t)^{q}(1-t)^{q(N+1)}}\\
    &\gtrsim\frac{\widehat{\om}(a)^q}{\widehat{\nu}(2|a|-1)^q(1-|a|)^{q(N+1)-1}}\\
    &\gtrsim\frac{\widehat{\om}(a)^q}{\widehat{\nu}(a)^q(1-|a|)^{q(N+1)-1}},
    \end{split}
    \end{equation}
which combined with \eqref{eq:tf1} gives the assertion in the case $0\le t\le |a|$ and $0<q\le2$. If $|a|\le t<1$, then the assertion readily follows by applying the estimate $\om_{t,j-N}\gtrsim\widehat{\om}(t)$ in the reasoning \eqref{eq:tf2}.

Let now $q>2$ (this approach actually works for any $q>1$). 
We begin by showing that
    \begin{equation}\label{2-1}
    \left\|\Delta_k F_{N+1}\right\|^q_{H^q}\gtrsim 2^{k[q(N+1)-1]},\quad k\ge\log_2N,
    \end{equation}
where $F_{N+1}$ is the function, analytic in the unit disc, defined by the Maclaurin series
		\begin{equation*}%\label{1-1}
		F_{N+1}(z) = \sum_{j=N}^\infty j(j-1)\cdots(j-N+1)z^j, \quad z \in \D,
		\end{equation*}
and thus
		\begin{equation}\label{1}
    \Delta_k F_{N+1}(z)=\sum_{j=2^k}^{2^{k+1}-1}j(j-1)\cdots(j-N+1)z^{j},\quad z\in\D.
    \end{equation}
By using the well known estimate $M_\infty(r,f)\lesssim(\r-r)^{-\frac1q}M_q(\r,f)$, valid for all $0<r<\r<1$, $0<q<\infty$ and $f\in\H(\D)$, and \cite[Lemma~10]{PelRathg}, we deduce
    \begin{equation*}
    \begin{split}
    M_\infty\left(1-\frac1{2^k},\Delta_kF_{N+1}\right)
    &\lesssim\left(\frac1{2^k}-\frac1{2^{k+1}-1}\right)^{-\frac{1}{q}}M_q\left(1-\frac{1}{2^{k+1}-1},\Delta_kF_{N+1}\right)\\
    &\asymp2^{\frac{k}{q}}\|\Delta_kF_{N+1}\|_{H^q},
    \end{split}
    \end{equation*}
which together with \eqref{1} gives
    \begin{equation*}
    \begin{split}
    \|\Delta_kF_{N+1}\|_{H^q}^q
    &\gtrsim2^{-k}M^q_\infty\left(1-\frac1{2^k},\Delta_kF_{N+1}\right)
    \asymp2^{-k}\left(\sum_{j=2^k}^{2^{k+1}-1}j^N\left(1-\frac1{2^k}\right)^{j}\right)^q\\
    &\geq 2^{Nkq-k}\left(2^{k}\left(1-2^{-k}\right)^{2^{k+1}}
    %\left[1-\left(1-2^{-k}\right)^{2^k}\right]
    \right)^q %\left(2^{k+1}-1-2^k\right)^q
    \asymp2^{k[q(N+1)-1]}.
    \end{split}
    \end{equation*}
Therefore \eqref{2-1} is now proved.

Let $0\le t\le |a|$. For $|a|\ge 1/2$, choose $k\in\N$ such that $1-2^{-k}\le |a|<1-2^{-k-1}$. Then \eqref{eq:tf0}, %the M.~Riesz projection theorem,
\cite[Lemma~D]{PelRatKernels} and \cite[Lemma~2.1]{PelSum14}  yield
    \begin{equation*}
    \begin{split}
    \widehat{\om}(t)^q M^q_q(t,T_\om\left((B^\nu_a)^N\right)
    &\gtrsim\int_0^{2\pi}\left|\sum_{j= 2^k}^{2^{k+1}-1}
    \frac{j(j-1)\cdots(j-N+1) e^{i\theta(j-N)}\overline{a}^{j}}{2\nu_{2j+1}}\om_{t,j-N}\right|^q\,d\t\\
    &\gtrsim \frac{|a|^{q2^{k+1}} \om^q_{t,2^{k+1}-N}}{ \nu_{2^{k+1}}^q}
    \int_0^{2\pi}\left|\sum_{j= 2^k}^{2^{k+1}-1}
    j(j-1)\cdots(j-N+1) e^{i\theta(j-N)}\right|^q\,d\t\\
    &\gtrsim\frac{\om^q_{|a|,2^{k+1}-N}} {\widehat{\nu}\left(1-\frac{1}{2^{k+1}}\right)^q}
    \left \|\Delta_k F_{n+1} \right\|^q_{H^q}
    \gtrsim\frac{\widehat{\om}(a)^q}{\widehat{\nu}(a)^q}2^{k[(N+1)q-1]}\\
    &\asymp\frac{\widehat{\om}(a)^q}{\widehat{\nu}(a)^q(1-|a|)^{q(N+1)-1}},
    \end{split}
    \end{equation*}
and thus the assertion for $0\le t\le |a|$ and $q>2$ is valid. The assertion in the case $t\ge|a|$ and $q>2$ readily follows by applying the estimate $\om_{t,2^{k+1}-N}\gtrsim \widehat{\om}(t)$ in the reasoning above.
\end{proof}

\begin{proposition}\label{Theorem:T_w-boundedpq3 weights}
Let $0<p\le q<\infty$, $\nu\in\DD$ and $\om,\eta$ radial weights.
If $T_\om: A^p_\nu \to L^q_\eta$ is bounded, then
    \begin{equation}\label{first-necessary-condition}
    \begin{split}
    \sup_{0\le r<1}\frac{\widehat{\om}(r)^q}{(1-r)^{\frac{q}{p}-1}\left(\int_r^1 s\nu(s)\,ds \right)^{\frac{q}{p}}}\int_0^r\frac{t\eta(t)}{\widehat{\om}(t)^q}\,dt
    \lesssim\|T_\om\|^q_{A^p_\nu\to L^q_\eta}<\infty
    \end{split}
    \end{equation}
and
    \begin{equation}\label{second-necessary-condition}
    \begin{split}
    \sup_{0\le r<1}\frac{\int_r^1 s\eta(s)\,ds)}{(1-r)^{\frac{q}{p}-1}\left(\int_r^1 s\nu(s)\,ds \right)^{\frac{q}{p}}}
    \lesssim\|T_\om\|^q_{A^p_\nu\to L^q_\eta}<\infty.
    \end{split}
    \end{equation}
\end{proposition}

\begin{proof}
Assume that $T_\om:A^p_\nu\to L^q_{\eta}$ is bounded. Then, for each $a\in\D$ and $N\in\N$,
    \begin{equation}\label{eq:i2}
    \begin{split}
    \left\| T_\om\left((B_a^\nu)^{(N)} \right)\right\|^q_{L^q_\eta}
    &\le\|T_\om\|^q_{A^p_\nu\to L^q_\eta}\left\|\left(B_a^\nu \right) ^{(N)} \right\|^q_{A^p_\nu}\\
    &\asymp\|T_\om\|^q_{A^p_\nu\to L^q_\eta}\left(\int_0^{|a|}\frac{dt}{\widehat{\nu}(t)^{p-1}(1-t)^{p(N+1)}}\right)^\frac{q}p,\quad |a|\to 1^-,
    \end{split}\end{equation}
by \cite[Theorem~1]{PelRatKernels}. If $p\ge 1$, then
    \begin{equation*}
    \left\|(B_a^\nu)^{(N)}\right\|^p_{A^p_\nu}
    \lesssim\frac{1}{\widehat{\nu}(a)^{p-1}}\int_0^{|a|}\frac{dt}{(1-t)^{p(N+1)}}
    \asymp\frac{1}{\widehat{\nu}(a)^{p-1}(1-|a|)^{p(N+1)-1}}, \quad |a|\to 1^-,
    \end{equation*}
for each $N\in\N$. If $0<p<1$, choose $N=N(p,\nu)\in\N$ such that $N>\frac{1+(1-p)\b}{p}-1$, where $\b=\b(\nu)>0$ is that of \cite[Lemma~2.1(ii)]{PelSum14}. Then
    \begin{equation*}
    \begin{split}
    \left\|(B_a^\nu)^{(N)}\right\|^p_{A^p_\nu}
    &\lesssim \frac{1}{\widehat{\nu}(a)^{p-1}(1-|a|)^{(1-p)\beta}} \int_0^{|a|}\frac{dt}{(1-t)^{p(N+1)-(1-p)\beta}}\\
    &\asymp \frac{1}{\widehat{\nu}(a)^{p-1}(1-|a|)^{p(N+1)-1}}, \quad |a|\to 1^-.
    \end{split}
    \end{equation*}
That is, for each $0<p<\infty$, there exists $N=N(p,\nu)\in\N$ such that
    \begin{equation}\label{eq:iii2}
    \left\|(B_a^\nu)^{(N)}\right\|^p_{A^p_\nu}\lesssim \frac{1}{\widehat{\nu}(a)^{p-1}(1-|a|)^{p(N+1)-1}},\quad |a|\to 1^-.
    \end{equation}
On the other hand, Lemma~\ref{le:testfunctionsq>p} yields
    \begin{equation}
    \begin{split}\label{eq:ii2}
    \left\| T_\om\left((B_a^\nu)^{(N)} \right)\right\|^q_{L^q_\eta}
    &\asymp\int_0^{|a|} M_q^q\left(t,T_\om\left(B_a^\nu\right)^{(N)}\right)t\eta(t)\,dt
    +\int_{|a|}^1M_q^q\left(t, T_\om\left(B_a^\nu\right)^{(N)}\right) t\eta(t)\,dt\\
    &\gtrsim\left(\int_0^{|a|}\frac{t\eta(t)}{\widehat{\om}(t)^q}\,dt\right)
    \frac{\widehat{\om}(a)^q}{\widehat{\nu}(a)^q(1-|a|)^{q(N+1)-1}}\\
    &\quad+\frac{\widehat{\eta}(a)}{\widehat{\nu}(a)^q(1-|a|)^{q(N+1)-1}},\quad |a|\ge 1-\frac{1}{2N}.
    \end{split}
    \end{equation}
The assertions follow by combining \eqref{eq:i2},\eqref{eq:iii2} and \eqref{eq:ii2}.
\end{proof}

It is worth noticing that replacing the  derivatives of the Bergman reproducing kernels  by the monomials in
the proof of Proposition~\ref{Theorem:T_w-boundedpq3 weights} yields analogous conditions to \eqref{first-necessary-condition} and \eqref{second-necessary-condition} but
without the factor $(1-r)^{\frac{q}{p}-1}$ in the denominator on the left hand side. So, one gets the same condition
for $q=p$ but a weaker one for $q>p$.

We make two observations on Proposition~\ref{Theorem:T_w-boundedpq3 weights}. First, \eqref{first-necessary-condition} implies \eqref{second-necessary-condition} if $\om\in\DD$ and $\eta\in\Dd$. Namely, by \cite[Lemma~2.1]{PelSum14} there exists $\alpha=\alpha(\om,q)>0$ such that $(1-t)^\alpha\widehat{\om}(t)^{-q}$ is essentially decreasing, and hence
    \begin{equation*}
    \int_0^r\frac{t\eta(t)}{\widehat{\om}(t)^q}\,dt
    \gtrsim\frac{1}{\widehat{\om}(r)^q}\int_0^r\left(\frac{1-r}{1-t}\right)^\a t\eta(t)\,dt.
    \end{equation*}
Let $K=K(\eta)>1$ be that of the definition of $\Dd$, and define $r_n=1-K^{-n}$ for all $n\in\N\cup\{0\}$. If $r_1\le r<1$, then there exists $m\in\N$ such that $r_m\le r<r_{m+1}$. Therefore
    \begin{equation*}
    \begin{split}
    \int_0^r\left(\frac{1-r}{1-t}\right)^\a t\eta(t)\,dt
    &\ge \int_{r_{m-1}}^{r_{m}}\left(\frac{1-r}{1-t}\right)^\a t\eta(t)\,dt
    \ge r_1\left(\frac{1-r_{m+1}}{1-r_{m-1}}\right)^\a\left(\widehat{\eta}(r_{m-1})-\widehat{\eta}(r_{m})\right)\\
    &\ge r_1\frac{(C-1)}{K^{2\a}}\widehat{\eta}(r_{m})\ge r_1\frac{(C-1)}{K^{2\a}}\widehat{\eta}(r),\quad r_1\le r<1,
    \end{split}
    \end{equation*}
for all $\alpha>0$.
So, if $\om\in\DD$ and $\eta\in\Dd$, then
	\begin{equation}\label{6new}
	\int_0^r\frac{t\eta(t)}{\widehat{\om}(t)^q}\,dt\gtrsim \frac{\widehat{\eta}(r)}{\widehat{\om}(r)^q}, \quad
	r_1(\eta)=r_1\le r<1,
	\end{equation}
and thus \eqref{second-necessary-condition} follows from \eqref{first-necessary-condition} in this case.

Second, if $\nu\in\DD$, the supremum in \eqref{second-necessary-condition} is finite if and only if $A^p_\nu$ is continuously embedded into $L^q_\eta$ by \cite[Theorem~1]{PelRat2014}, and thus $\|I_d\|_{A^p_\nu\to L^q_\eta}\lesssim\|T_\om\|_{A^p_\nu\to L^q_\eta}$ for each radial weight $\om$.

Bearing in mind the special case $\eta=\nu$ of \eqref{second-necessary-condition}, we get the following immediate consequence.

\begin{corollary}\label{co:7}
Let $0<p<q<\infty$ and $\nu\in\DD$. Then $T_\om:A^p_\nu\to L^q_\nu$ is unbounded for each radial weight $\om$.
\end{corollary}

\subsection{Proof of Theorem~\ref{Theorem:T_w-bounded-intro}}

Let $h^\star(z)=\esssup_{0<r<|z|}|h(rz/|z|)|$ denote the radial maximal function of $h$ at $z\in\D\setminus\{0\}$.

\begin{lemma}\label{Lemma:radial-maximal-function-weighted}
Let $0<p\le1\le q<\infty$ and $\om$ a radial weight. Then there exist constants $C_1=C_1(p,\om)>0$ and $C_2=C_2(q,\om)>0$ such that
    $$
    \left(\int_r^1h(te^{i\t})\om(t)\,dt\right)^p\le C_1\int_r^1 h^\star(te^{i\t})^p\widehat{\om}(t)^{p-1}\om(t)t\,dt
    $$
and
    $$
    \int_r^1h(te^{i\t})^q\widehat{\om}(t)^{q-1}\om(t)\,dt
    \le C_2\left(\int_r^1h^\star(te^{i\t})\om(t)t\,dt\right)^q
    $$
for all non-negative measurable functions $h$ on $\D$ and $\t\in\mathbb{R}$.
\end{lemma}

\begin{proof}
For $1<K<\infty$ and $0\le r<1$, define $\r_n=\r_n(\om,K,r)=\min\{0\le t<1:\widehat{\om}(t)=\widehat{\om}(r)K^{-n}\}$ for all $n\in\N$, and set $\r_0=r$. Then $\{\r_n\}$ is increasing such that $\r_n\to1^-$, as $n\to\infty$, and
    $$
    \int_{\r_n}^{\r_{n+1}}\om(t)\,dt=\widehat{\om}(\r_n)-\widehat{\om}(\r_{n+1})=\widehat{\om}(\r_j)K^{j-n}\left(\frac{K-1}{K}\right),\quad n,j\in\N\cup\{0\}.
    $$
Hence
    \begin{equation*}
    \begin{split}
    \left(\int_r^1h(te^{i\t})\om(t)\,dt\right)^p
    &=\left(\sum_{n=0}^\infty\int_{\r_n}^{\r_{n+1}}h(te^{i\t})\om(t)\,dt\right)^p\\
    &\le\left(\sum_{n=0}^\infty h^\star(\r_{n+1}e^{i\t})(\widehat{\om}(\r_n)-\widehat{\om}(\r_{n+1}))\right)^p\\
    &=(K-1)^p\left(\sum_{n=0}^\infty h^\star(\r_{n+1}e^{i\t})\widehat{\om}(\r_{n+1})\right)^p\\
    &\le(K-1)^p\sum_{n=0}^\infty h^\star(\r_{n+1}e^{i\t})^p\widehat{\om}(\r_{n+1})^p
    \frac{\widehat{\om}(\r_{n+1})-\widehat{\om}(\r_{n+2})}{\widehat{\om}(\r_{n+1})\frac{K-1}{K}}\\
    &=K(K-1)^{p-1}\sum_{n=0}^\infty h^\star(\r_{n+1}e^{i\t})^p\widehat{\om}(\r_{n+1})^{p-1}
    (\widehat{\om}(\r_{n+1})-\widehat{\om}(\r_{n+2}))\\
    &\le K(K-1)^{p-1}\sum_{n=0}^\infty\int_{\r_{n+1}}^{\r_{n+2}} h^\star(te^{i\t})^p\widehat{\om}(t)^{p-1}\om(t)\,dt\\
    &\le\frac{K(K-1)^{p-1}}{\r_1}\int_r^{1}h^\star(te^{i\t})^p\widehat{\om}(t)^{p-1}\om(t)t\,dt.
    \end{split}
    \end{equation*}
In a similar manner one deduces
    \begin{equation*}
    \begin{split}
    \int_r^1h(te^{i\t})^q\widehat{\om}(t)^{q-1}\om(t)\,dt
    &= \sum_{n=0}^\infty \int_{\r_n}^{\r_{n+1}} h(te^{i\t})^q\widehat{\om}(t)^{q-1}\om(t)\,dt \\
    &\le \sum_{n=0}^\infty h^\star(\r_{n+1}e^{i\t})^q\widehat{\om}(\r_n)^{q-1}\int_{\r_n}^{\r_{n+1}}\om(t)\,dt \\
    &= \frac{K-1}K \sum_{n=0}^\infty h^\star(\r_{n+1}e^{i\t})^q\widehat{\om}(\r_n)^q \\
    &\le \frac{K-1}K \left(\sum_{n=0}^\infty h^\star(\r_{n+1}e^{i\t})\widehat{\om}(\r_n)\right)^q \\
    &=\frac{K-1}K \left(\sum_{n=0}^\infty h^\star(\r_{n+1}e^{i\t}) \int_{\r_{n+1}}^{\r_{n+2}}\om(t)\,dt \frac{\widehat{\om}(\r_n)}{\widehat{\om}(\r_n)\frac{K-1}{K^2}}\right)^q \\
    &\le \frac{(K-1)^{1-q}}{K^{1-2q}} \left(\sum_{n=0}^\infty \int_{\r_{n+1}}^{\r_{n+2}}h^\star(te^{i\t})\om(t)\,dt\right)^q \\
    &\le K^{2q-1}(K-1)^{1-q} \left(\frac{1}{\r_1}\int_{\r_1}^1h^\star(te^{i\t})\om(t)t\,dt\right)^q \\
    &\le \frac{K^{2q-1}(K-1)^{1-q}}{\r_1^q}\left(\int_r^1h^\star(te^{i\t})\om(t)t\,dt\right)^q,
    \end{split}
    \end{equation*}
and thus the lemma is proved.
\end{proof}

The next lemma shows that the condition $\DD_p(\om,\nu)<\infty$ is self-improving in the sense that if it is satisfied for some $p>0$, then it is also satisfied when $p$ is replaced by a slightly smaller number.

\begin{lemma}\label{Lemma:self-improving-radial-w}
Let $0<p<\infty$ and $\nu,\om$ radial weights on $\D$. Then
    $$
    \DD_p(\om,\nu)\le\DD_{p-\e}(\om,\nu)\le\frac{p}{p-\e(1+\DD_p(\om,\nu))}\DD_p(\om,\nu)
    $$
for all $\e\in\left(0,\frac{p}{\DD_p(\om,\nu)+1}\right)$.
\end{lemma}

\begin{proof}
The first inequality is obvious.
Let us denote $\nu_1(t)=t\nu(t)$.
 An integration by parts gives
    \begin{equation*}
    \begin{split}
    \int_0^r\frac{\nu_1(t)}{\widehat{\om}(t)^{p-\e}}\,dt
    &= \widehat{\om}(r)^\e \int_0^r\frac{\nu_1(t)}{\widehat{\om}(t)^p}\,dt
    +\e\int_0^r\left(\int_0^s\frac{\nu_1(t)}{\widehat{\om}(t)^{p}}\,dt\right)\,\widehat{\om}(s)^{\e-1}\om(s)\,ds\\
    &\le\DD_p(\om,\nu)\frac{\widehat{\nu_1}(r)}{\widehat{\om}(r)^{p-\e}}
    +\e\DD_p(\om,\nu)\int_0^r\frac{\widehat{\nu_1}(s)}{\widehat{\om}(s)^{p+1-\e}}\om(s)\,ds,
    \end{split}
    \end{equation*}
and since Fubini's theorem yields
    \begin{equation*}
    \begin{split}
    \int_0^r\frac{\widehat{\nu_1}(s)}{\widehat{\om}(s)^{p+1-\e}}\om(s)\,ds
    &=\int_0^r\frac{\int_s^1\nu_1(t)\,dt}{\widehat{\om}(s)^{p+1-\e}}\om(s)\,ds\\
    &=\int_0^r\frac{\int_s^r\nu_1(t)\,dt}{\widehat{\om}(s)^{p+1-\e}}\om(s)\,ds
    +\widehat{\nu_1}(r)\int_0^r\frac{\om(s)}{\widehat{\om}(s)^{p+1-\e}}\,ds\\
    &=\int_0^r\nu_1(t)\left(\int_0^t\frac{\om(s)}{\widehat{\om}(s)^{p+1-\e}}\,ds\right)\,dt
    +\widehat{\nu_1}(r)\int_0^r\frac{\om(s)}{\widehat{\om}(s)^{p+1-\e}}\,ds\\
    &=\frac{1}{p-\e}\left(\int_0^r\frac{\nu_1(t)}{\widehat{\om}(t)^{p-\e}}\,dt
    -\frac{1}{\widehat{\om}(0)^{p-\e}}\int_0^r\nu_1(t)\,dt
    +\frac{\widehat{\nu_1}(r)}{\widehat{\om}(r)^{p-\e}}
    -\frac{\widehat{\nu_1}(r)}{\widehat{\om}(0)^{p-\e}}\right)\\
    &\le\frac{1}{p-\e}\left(\int_0^r\frac{\nu_1(t)}{\widehat{\om}(t)^{p-\e}}\,dt
    +\frac{\widehat{\nu_1}(r)}{\widehat{\om}(r)^{p-\e}}\right),
    \end{split}
    \end{equation*}
we deduce
    $$
    \int_0^r\frac{\nu_1(t)}{\widehat{\om}(t)^{p-\e}}\,dt
    \le\frac{\DD_p(\om,\nu)\left(1+\frac{\e}{p-\e}\right)}{1-\frac{\DD_p(\om,\nu)\e}{p-\e}}\frac{\widehat{\nu_1}(r)}{\widehat{\om}(r)^{p-\e}}
    =\frac{p\DD_p(\om,\nu)}{p-\e(1+\DD_p(\om,\nu))}\frac{\widehat{\nu_1}(r)}{\widehat{\om}(r)^{p-\e}}
    $$
for $\e\in\left(0,\frac{p}{\DD_p(\om,\nu)+1}\right)$. The assertion follows.
\end{proof}

\par With these preparations we are ready for the proof.

\medskip

\begin{Prf}{\em{Theorem~\ref{Theorem:T_w-bounded-intro}.}}
Obviously, (i) implies (ii), and (ii) implies (iii)   together with the inequality
 $
    \DD_p(\om,\nu)\lesssim\|T_\om\|^p_{A^p_\nu\to L^p_\nu}
    $
follows from
 the case $q=p$ of Proposition~\ref{Theorem:T_w-boundedpq3 weights}.

By \cite[Theorem~1]{PelRat2014}, $\mu_{p,\om,\nu}$ is a $q$-Carleson measure for $A^q_\nu$ if and only if $\mu_{p,\om,\nu}(S(a))\lesssim\nu(S(a))$ for all Carleson squares $S(a)=\{z\in\D:|\arg z-\arg a|<\frac{1-|a|}{2},\,|z|\ge 1-|a|\}$ with $a\in\D\setminus\{0\}$. But Fubini's theorem yields
    \begin{equation}\label{CM-T_w}
    \begin{split}
    \int_{|a|}^1\widehat{\om}(t)^{p-1}\om(t)\int_0^{t}\frac{\nu(s)s\,ds}{\widehat{\om}(s)^p}\,dt
    &=\int_{|a|}^1\widehat{\om}(t)^{p-1}\om(t)
    \left(\int_0^{{|a|}}+\int_{|a|}^{t}\right)\frac{\nu(s)s\,ds}{\widehat{\om}(s)^p}\,dt\\
    &=\left(\int_{|a|}^1\widehat{\om}(t)^{p-1}\om(t)\,dt\right)\int_0^{|a|}\frac{\nu(s)s\,ds}{\widehat{\om}(s)^p}\\
    &\quad+\int_{|a|}^1\frac{\nu(s)s}{\widehat{\om}(s)^p}\int_s^1\widehat{\om}(t)^{p-1}\om(t)\,dt\,ds\\
    &=\frac{\widehat{\om}(a)^p}{p}\int_0^{|a|}\frac{\nu(s)s\,ds}{\widehat{\om}(s)^p}+\frac{1}{p}\int_{|a|}^1\nu(s)s\,ds,
    \end{split}
    \end{equation}
and it follows that (iii) and (iv) are equivalent.

Assume (iii), and let first $0<p\le1$. By Lemma~\ref{Lemma:radial-maximal-function-weighted} and Fubini's theorem,
    \begin{equation*}
    \begin{split}
    \int_0^1\left(T_\om^N(f)(re^{i\theta})\right)^p\nu(r)r\,dr
    &=\int_0^1\left(\int_r^1N(f)(te^{i\theta})\om(t)\,dt\right)^p\frac{\nu(r)r}{\widehat{\om}(r)^p}\,dr\\
    &\lesssim\int_0^1\left(\int_r^1(N(f))^\star(te^{i\theta})^p\widehat{\om}(t)^{p-1}\om(t)t\,dt\right)\frac{\nu(r)r}{\widehat{\om}(r)^p}\,dr\\
    &=\int_0^1N(f)^p(te^{i\theta})\widehat{\om}(t)^{p-1}\om(t)\left(\int_0^t\frac{\nu(r)r}{\widehat{\om}(r)^p}\,dr\right)t\,dt.
    \end{split}
    \end{equation*}
An integration with respect to $\theta$ and the Hardy-Littlewood maximal theorem~\cite[Theorem~3.1, p.~57]{Garnett1981}
(or \cite[Lemma~4.4]{PR2014Memoirs}) together with (iii) now yield
    \begin{equation*}
    \begin{split}
    \|T_\om^N(f)\|_{L^p_\nu}^p&\lesssim\int_{\D}|f(z)|^p\widehat{\om}(z)^{p-1}\om(z)\left(\int_0^{|z|}\frac{\nu(r)r}{\widehat{\om}(r)^p}\,dr\right)\,dA(z)\\
    &=\int_\D|f(z)|^p\,d\mu_{p,\om,\nu}(z)\le\|Id\|_{A^p_\nu\to L^p_{\mu_{p,\om,\nu}}}^p\|f\|_{A^p_\nu}^p.
    \end{split}
    \end{equation*}
Thus $T_\om^N:A^p_\om\to L^p_\om$ is bounded. Moreover, the estimates above together with \cite[Theorem~3]{PelRatSie2015} and \eqref{CM-T_w} give
    $$
    \|T\|^p_{A^p_\nu\to L^p_\nu}\le\|T^N\|^p_{A^p_\nu\to L^p_\nu}\lesssim\|Id\|_{A^p_\nu\to L^p_{\mu_{p,\om,\nu}}}^p\asymp\DD_p(\om,\nu),
    $$
provided $0<p\le1$.

Let now $1<p<\infty$, and fix $\e=\e(p,\om,\nu)\in\left(0,\min\left\{p-1,\frac{p}{\DD_p(\om,\nu)}\right\}\right)$. Define $h(z)=\om(z)^{\frac1{p'}}\widehat{\om}(z)^\frac{-p+1+\e}{p}$ for all $z\in\D$, and set $\Omega=\{z\in\D:\om(z)>0\}$. Then H\"older's inequality, Fubini's theorem and the Hardy-Littlewood maximal theorem yield
    \begin{equation*}
    \begin{split}
    \|T_\om^N(f)\|_{L^p_\nu}^p
    &=\int_\D\left(\int_{|z|}^1N(f)\left(t\frac{z}{|z|}\right)\om(t)\chi_{\Omega}(t)\,dt\right)^p\frac{\nu(z)}{\widehat{\om}(z)^p}\,dA(z)\\
    &=\int_\D\left(\int_{|z|}^1N(f)\left(t\frac{z}{|z|}\right)\frac{\om(t)}{h(t)}h(t)\chi_{\Omega}(t)\,dt\right)^p\frac{\nu(z)}{\widehat{\om}(z)^p}\,dA(z)\\
    &\le\int_\D
    \left(\int_{|z|}^1N^p(f)\left(t\frac{z}{|z|}\right)\left(\frac{\om(t)}{h(t)}\right)^p\chi_{\Omega}(t)\,dt\right)
    \left(\int_{|z|}^1h(s)^{p'}\,ds\right)^{p-1}\frac{\nu(z)}{\widehat{\om}(z)^p}\,dA(z)\\
    &\asymp \int_0^1 \int_0^{2\pi} \left(\int_{r}^1 N^p(f)\left(te^{i\theta}\right)\left(\frac{\om(t)}{h(t)}\right)^p\chi_{\Omega}(t)\,dt\right)
    \left(\int_{r}^1h(s)^{p'}\,ds\right)^{p-1}\frac{\nu(r)}{\widehat{\om}(r)^p}
    \,r d\theta dr\\
    &=\int_0^1 \int_0^{2\pi}
    N^p(f)\left(te^{i\theta}\right)\,d\t \left(\frac{\om(t)}{h(t)}\right)^p\chi_{\Omega}(t)
    \int_0^t\left(\int_{r}^1h(s)^{p'}\,ds\right)^{p-1}\frac{\nu(r)}{\widehat{\om}(r)^p}r\,dr\,\,dt\\
    &\lesssim
    \int_0^1\int_0^{2\pi}|f(te^{i\t})|^p\,d\t \,\left(\frac{\om(t)}{h(t)}\right)^p\chi_{\Omega}(t)\int_0^t\left(\int_{r}^1h(s)^{p'}\,ds\right)^{p-1}\frac{\nu(r)}{\widehat{\om}(r)^p}r\,dr\,dt\\
    &\asymp
    \int_\D|f(z)|^p\left(\frac{\om(z)}{h(z)}\right)^p\chi_{\Omega}(z)\int_0^{|z|}\left(\int_{r}^1h(s)^{p'}\,ds\right)^{p-1}\frac{\nu(r)}{\widehat{\om}(r)^p}r\,dr\,dA(z)\\
    &=\left(\frac{p-1}{\e}\right)^{p-1}\int_\D|f(z)|^p\om(z)\widehat{\om}(z)^{p-1-\e}\int_0^{|z|}\frac{\nu(r)}{\widehat{\om}(r)^{p-\e}}r\,dr\,dA(z)\\
    &\lesssim\int_\D|f(z)|^p\,d\mu_{p-\e,\om,\nu}(z).
    \end{split}
    \end{equation*}
Now that $\mu_{p-\e,\om,\nu}$ is a $p$-Carleson measure for $A^p_\om$ by \eqref{CM-T_w} and Lemma~\ref{Lemma:self-improving-radial-w}, it follows that $T_\om^N:A^p_\nu\to L^p_\nu$ is bounded. Moreover, a reasoning similar to that in the case $0<p\le1$ together with Lemma~\ref{Lemma:self-improving-radial-w} gives
    $$
    \|T_\om\|^p_{A^p_\nu\to L^p_\nu}\le\|T_\om^N\|^p_{A^p_\nu\to L^p_\nu}\lesssim\|Id\|_{A^p_\nu\to L^p_{\mu_{p-\e,\om,\nu}}}^p\asymp\DD_{p-\e}(\om,\nu)\asymp\DD_p(\om,\nu).
    $$
The proof is complete.
\end{Prf}

\medskip

\begin{Prf} \emph{Corollary~\ref{corollary:A^p-L^p-the same}.}
Since $\om,\nu\in\R$, the result follows joining \cite[Theorem 11]{PelRatKernels}, Theorem~\ref{Theorem:T_w-bounded-L^p-case} and Theorem~\ref{Theorem:T_w-bounded-intro}.
\end{Prf}

\section{Weak type inequalities}

We use the conventions $0\cdot\infty=0$ and $1/0=\infty$. The next proof follows ideas from \cite{AndersenMuckStudia82}.

\medskip

\begin{Prf} \emph{Theorem~\ref{Theorem:T_w-bounded-weak-L^p-case}.}
Assume first (i), that is,
    \begin{equation}\label{5}
    \la^p \eta(\{z \in \D : |T_\om(f)(z)| > \la\}) \lesssim \|f\|^p_{L^p_\nu}, \quad \la > 0, \quad f \in L^p_\nu.
    \end{equation}
Let $h(r)=\left(\int_r^1\left(\frac{\om(s)}{s\nu(s)}\right)^{p'}s\nu(s)\,ds\right)^{\frac1{p'}}$ and denote
    $$
    N(t,r)=\left(\frac{\int_t^r \eta(s)s\,ds}{\widehat{\om}(t)^p}\right)^\frac1p
    \left(\int_r^1 \left(\frac{\om(s)}{s\nu(s)}\right)^{p'}s\nu(s)\,ds\right)^{\frac1{p'}},\quad 0\le t<r<1,
    $$
for short. Let us rule out the trivial cases $h(r)=0$ or $h(r)=\infty$ first. If $h(r)=0$, then $N(t,r)=0$ for all $0\le t<r$ by the convention. If $h(r)=\infty$, then $\left( \om (s\nu)^{-1}\right)^{1/p}\notin L^{p'}_\om(r,1)$, and hence there exists $g\in L^{p}_\om(r,1)$ such that $g\left( \om (s\nu)^{-1}\right)^{1/p}\notin L^{1}_\om(r,1)$. Define $f(z)=g(|z|)\left(\om(z)(|z|\nu(z))^{-1}\right)^{\frac1p}$ for $r\le|z|<1$, and let $f\equiv0$ elsewhere on $\D$. Then $T_\om (f)(z)=\infty$ for $|z|\le r$, and therefore the weak inequality \eqref{5} yields
    \begin{equation*}
    \begin{split}
    \int_0^r s\eta(s)\,ds
    &\le\int_{\left\{z\in\D: |T_\om(f)(z)|>\lambda \right\}} \eta(z)\,dA(z)\\
    &\lesssim\frac{1}{\lambda^p}\int_{\left\{z\in\D:\, r\le |z|<1\right\}} |f(z)|^p\nu(z)\,dA(z)
    =\frac{2}{\lambda^p} \int_r^1 |g(s)|^p\om(s)\,ds,\quad \lambda>0.
    \end{split}
    \end{equation*}
By letting $\lambda\to\infty$, we deduce $\int_0^r\eta(s)s\,ds=0$, and thus $N(t,r)=0$ for all $0\le t<r$ in this case also.

Assume that $0<h(r)<\infty$, and let $f_r(z)=\left(\frac{\om(z)}{|z|\nu(z)}\right)^{p'/p}\chi_{\D\setminus D(0,r)}(z)$, $z\in\D$, and
    $$
    \la_{r,t} = \frac{\int_r^1 \left(\frac{\om(s)}{s\nu(s)}\right)^{p'}s\nu(s)\,ds}{\widehat{\om}(t)},\quad0\le t\le r<1.
    $$
Then
    $$
    \|f_r\|^p_{L^p_\nu} = 2\int_r^1\left(\frac{\om(s)}{s\nu(s)}\right)^{p'}\nu(s)s\,ds
    $$
and
    \begin{equation*}
    \begin{split}
    T_\om(f_r)(z)
    &= \widehat{\om}(z)^{-1} \int_{|z|}^1 \chi_{[r,1)}(s)\left(\frac{\om(s)}{s\nu(s)}\right)^{p'/p}\om(s)\,ds \\
    &= \widehat{\om}(z)^{-1} \int_r^1 \left(\frac{\om(s)}{s\nu(s)}\right)^{p'}s\nu(s)\,ds \ge \la_{r,t},
    \quad  0 \leq t < |z| \leq r < 1.
    \end{split}
    \end{equation*}
Therefore
    \begin{equation*}
    \begin{split}
    \int_t^r \eta(s)s\,ds
    &\le2\eta(\{z \in \D : |T_\om(f_r)(z)|\ge\la_{r,t}\})
    \le2\|T_\om\|_{L^p_\nu \to L^{p,\infty}_\eta}^p\frac{\|f_r\|^p_{L^p_\nu}}{\la_{r,t}^p} \\
    &\asymp\|T_\om\|_{L^p_\nu \to L^{p,\infty}_\eta}^p \frac{\widehat{\om}(t)^p \int_r^1 \left(\frac{\om(s)}{s\nu(s)}\right)^{p'}s\nu(s)\,ds}{\left(\int_r^1 \left(\frac{\om(s)}{s\nu(s)}\right)^{p'}s\nu(s)\,ds\right)^p} \\
    &=\|T_\om\|_{L^p_\nu \to L^{p,\infty}_\eta}^p\frac{\widehat{\om}(t)^p}{\left(\int_r^1 \left(\frac{\om(s)}{s\nu(s)}\right)^{p'}s\nu(s)\,ds\right)^{\frac{p}{p'}}},
    \end{split}
    \end{equation*}
and thus
    $$
   \frac{\int_t^r \eta(s)s\,ds}{\widehat{\om}(t)^p} \left(\int_r^1 \left(\frac{\om(s)}{s\nu(s)}\right)^{p'}s\nu(s)\,ds\right)^{p/p'}
   \lesssim\|T_\om\|_{L^p_\nu \to L^{p,\infty}_\eta}^p, \quad 0 \leq t < r < 1.
   $$
This implies (ii) and $N_{p}(\om,\nu,\eta)\lesssim\|T_\om\|_{L^p_\nu\to L^{p,\infty}_\eta}$.

Assume next (ii). Then, for each $\e>0$, we have
 \begin{equation*}
    \begin{split}
    N_{p}^p(\om,\nu,\eta)\frac{1}{\e}\left(\frac{1}{\widehat{\om}(r)^\ep}- \frac{1}{\widehat{\om}(0)^\ep} \right)
    &=N_{p}^p(\om,\nu,\eta)\int_0^r \frac{\om(t)}{\widehat{\om}(t)^{1+\ep}}\,dt\\
    &\ge h^p(r)\int_0^r\left(\int_t^r\eta(s)s\,ds\right)\frac{\om(t)}{\widehat{\om}(t)^{1+\ep+p}}\,dt\\
    &=h^p(r)\int_0^r\left(\int_0^s \frac{\om(t)}{\widehat{\om}(t)^{1+\ep+p}}\,dt\right)\eta(s)s\,ds\\
    &=\frac{h^p(r)}{p+\ep}\int_0^r\left(\frac{1}{\widehat{\om}(s)^{\ep+p}}-\frac{1}{\widehat{\om}(0)^{\ep+p}}\right)\eta(s)s\,ds,
    \end{split}
    \end{equation*}
from which (iii) and the inequality $N_{p}(\om,\nu,\eta)\gtrsim M_{p,\e}(\om,\nu,\eta)$ for each fixed $\e>0$ follow.

Assume now (iii), and let $\e>0$. Let $f$ be a compactly supported non-negative step function on $\D$, that is, $f=\sum_{j=1}^l P_j\chi_{R_j}$,
where $P_j\ge 0$ and $R_j=\{re^{i\theta}:0\le A_j\le r\le B_j<1,\, c_j\le\theta\le d_j  \}$ with $d_j-c_j\le2\pi$. Define $E_\theta(\lambda)=\{r \in (0,1):T_\om(f)(re^{i\theta})>\lambda\}$ for any $\theta\in[0,2\pi)$ and $\lambda>0$, and $H(r)=\int_0^r\frac{s\eta(s)}{\widehat{\om}(s)^{p+\ep}}\,ds$ for all $0\le r<1$.
If $r_0=\inf\{r\in(0,1):h(r)<\infty\}>0$, then $\eta(r)=0$ almost everywhere on $[0,r_0]$. Therefore $\int_{E_\theta(\lambda)}\eta(s)\,ds = \int_{E_\theta(\lambda)\cap(r_0,1)}\eta(s)\,ds$, where
    $$
    E_\theta(\lambda)\cap(r_0,1) = \bigcup_{k=1}^n (a_k,b_k),
    \quad r_0 \leq a_1 < \cdots < b_k \leq a_{k+1} < \cdots < b_n < 1.
    $$
Let $r\in [a_k,b_k]$ for some $k=1,\ldots,n$. Then $\widehat{\om}(r)\le\lambda^{-1}\int_r^1f(se^{i\theta})\om(s)\,ds$, and H\"older's inequality yields
    \begin{equation}\label{Eq:weak-ineq-what-estimate}
    \begin{split}
    \widehat{\om}(r)^p
    &\le\lambda^{-p} \left(\int_{r}^{1}f(se^{i\theta})h(s)^{1/p}(s\nu(s))^{1/p}h(s)^{-1/p}\frac{\om(s)}{(s\nu(s))^{1/p}}\,ds\right)^p \\
    &\le\lambda^{-p} \int_{r}^{1} f(se^{i\theta})^p h(s)s\nu(s)\,ds \left(\int_{r}^{1} h(s)^{-\frac{1}{p-1}}\left(\frac{\om(s)}{s\nu(s)}\right)^{p'}s\nu(s)\,ds\right)^{p-1} \\
    &= \lambda^{-p} \int_{r}^{1}f(se^{i\theta})^p h(s)s\nu(s)\,ds \left(-p'\,[h(s)]_{s=r}^1\right)^{p-1} \\
    &=(p')^{p-1} \lambda^{-p} h(r)^{p-1} \int_{r}^{1}f(se^{i\theta})^p h(s)s\nu(s)\,ds.
    \end{split}
    \end{equation}
This inequality will be repeatedly used throughout the rest of the proof.

We now proceed to estimate $\int_{E_\theta(\lambda)}s\eta(s)\,ds$. An integration by parts, the assumption $M_{p,\e}<\infty$
and another integration by parts yield
    \begin{equation*}
    \begin{split}
    \int_{a_k}^{b_k} s\eta(s)\,ds
    &= \int_{a_k}^{b_k} \widehat{\om}(s)^{p+\e} \frac{s\eta(s)}{\widehat{\om}(s)^{p+\e}}\,ds \\
    & =\left[ \widehat{\om}(s)^{p+\e}H(s)\right]^{b_k}_{a_k}
    +(p+\e)\int_{a_k}^{b_k} \widehat{\om}(s)^{p+\e-1}\om(s)H(s)\,ds
    \\ & \le \left[ \widehat{\om}(s)^{p+\e}H(s)\right]^{b_k}_{a_k}
    +(p+\e)M^p_{p,\e} \int_{a_k}^{b_k} \widehat{\om}(s)^{p-1}\om(s)h^{-p}(s)\,ds
    \\ & = \left[ \widehat{\om}(s)^{p+\e}H(s)
    -\frac{(p+\e)}{p}M^p_{p,\e}\widehat{\om}(s)^{p}h^{-p}(s)\right]^{b_k}_{a_k}\\
    &\quad-(p+\e)M^p_{p,\e}\int_{a_k}^{b_k} \widehat{\om}(s)^{p}h^{-p-1}(s)h'(s)\,ds
    \end{split}
    \end{equation*}
for each $k=1,\ldots,n$. Thus $\int_{E_\theta(\lambda)}\eta(s)\,ds \leq S_{1,\theta} + S_{2,\theta}$, where
    \begin{equation}\label{eq:s1}
    S_{1,\theta} = \sum_{k=1}^n\left[ \widehat{\om}(s)^{p+\e}H(s)
    -\frac{(p+\e)}{p}M^p_{p,\e}\widehat{\om}(s)^{p}h^{-p}(s)\right]^{b_k}_{a_k}
    \end{equation}
and
    \begin{equation}\label{eq:s2}
    S_{2,\theta} = -(p+\e)M_{p,\e}^p \sum_{k=1}^n \int_{a_k}^{b_k} \widehat{\om}(s)^{p}h^{-p-1}(s)h'(s)\,ds .
     \end{equation}
The sums $S_{1,\theta}$ and $S_{2,\theta}$ will now be considered separately. Since $H$ is nondecreasing,
    \begin{equation}
    \begin{split}\label{eq:weak0}
    &\sum_{k=1}^n\left[ \widehat{\om}(s)^{p+\e}H(s)\right]^{b_k}_{a_k}\\
    &=-\widehat{\om}(a_1)^{p+\e}H(a_1)
    +\sum_{k=1}^{n-1}\left(\widehat{\om}(b_k)^{p+\e}H(b_k)- \widehat{\om}(a_{k+1})^{p+\e}H(a_{k+1}) \right)
    +\widehat{\om}(b_n)^{p+\e}H(b_n)
    \\ & \le \sum_{k=1}^{n-1}\left(\widehat{\om}(b_k)^{p+\e}- \widehat{\om}(a_{k+1})^{p+\e} \right)H(b_k)
    +\widehat{\om}(b_n)^{p+\e}H(b_n),
    \end{split}
    \end{equation}
where the first negative term was simply discarded. Next, observe that for any $N>M>0$ and $a>0$, the function $g(x)=M(a^N-x^N)-Na^{N-M}(a^M-x^M)$ is nondecreasing on $[0,a]$. Hence
    $$
    \widehat{\om}(b_k)^{p+\e}- \widehat{\om}(a_{k+1})^{p+\e}\le \frac{p+\ep}{p}  \widehat{\om}(b_k)^{\e}
    \left( \widehat{\om}(b_k)^{p}- \widehat{\om}(a_{k+1})^{p}\right),\quad k=1,\dots,n-1.
    $$
This together with \eqref{eq:weak0}  implies
    \begin{equation*}
    \begin{split}
    \sum_{k=1}^n\left[ \widehat{\om}(s)^{p+\e}H(s)\right]^{b_k}_{a_k}
    &\le\frac{p+\e}{p} \sum_{k=1}^{n-1}\left(\widehat{\om}(b_k)^{p}- \widehat{\om}(a_{k+1})^{p} \right)\widehat{\om}(b_k)^{\e} H(b_k)
    +\widehat{\om}(b_n)^{p+\e}H(b_n)\\
    &\le M^p_{p,\e}\frac{p+\e}{p} \sum_{k=1}^{n-1}\left(\widehat{\om}(b_k)^{p}- \widehat{\om}(a_{k+1})^{p} \right) h^{-p}(b_k)
    + M^p_{p,\e}\widehat{\om}(b_n)^{p}h^{-p}(b_n).
    \end{split}
    \end{equation*}
By joining this inequality with \eqref{eq:s1}, we deduce
    \begin{equation}
    \begin{split}\label{eq:s1prima}
    S_{1,\theta}
    &\le M^p_{p,\e}\frac{p+\e}{p}\left(\widehat{\om}(a_{1})^{p}h^{-p}(a_1)+
    \sum_{k=1}^{n-1}\left(h^{-p}(a_{k+1})-  h^{-p}(b_k)\right)  \widehat{\om}(a_{k+1})^{p}\right).
    \end{split}
    \end{equation}
Next, use the monotonicity of the auxiliary function $g$ to deduce
    $$
    h^{-p}(a_{k+1})-  h^{-p}(b_k)\le ph^{-p+1}(a_{k+1})\left[ h^{-1}(a_{k+1})-  h^{-1}(b_k)\right], \quad k=1,\dots,n-1,
    $$
which together with \eqref{eq:s1prima} and \eqref{Eq:weak-ineq-what-estimate} yields
    \begin{equation}
    \begin{split}\label{eq:s1primaa}
    S_{1,\theta}
    &\le M^p_{p,\e}\frac{p+\e}{p}\left(\widehat{\om}(a_{1})^{p}h^{-p}(a_1)+
    p\sum_{k=1}^{n-1}\left(h^{-1}(a_{k+1})-  h^{-1}(b_k)\right) h^{1-p}(a_{k+1}) \widehat{\om}(a_{k+1})^{p}\right)\\
    &\le M^p_{p,\e}\frac{p+\e}{p}(p')^{p-1}\lambda^{-p}
    h^{-1}(a_1)\int_{a_1}^1f(se^{i\theta})^p h(s)s\nu(s)\,ds\\
    &\quad+M^p_{p,\e}(p+\e)(p')^{p-1}\lambda^{-p}\sum_{k=1}^{n-1}\left(h^{-1}(a_{k+1})-  h^{-1}(b_k)\right)\int_{a_{k+1}}^1f(se^{i\theta})^p h(s)s\nu(s)\,ds.
    \end{split}
    \end{equation}
To estimate the sum $S_{2,\theta}$, use \eqref{Eq:weak-ineq-what-estimate} and integrate by parts to obtain
    \begin{equation*}
    \begin{split}
    \frac{ S_{2,\theta}}{(p+\e)M^p_{p,\e}}
    &=\sum_{k=1}^n\int_{a_k}^{b_k} \widehat{\om}(s)^p h(s)^{-p-1}(-h'(s))\,ds\\
    &\le(p')^{p-1} \lambda^{-p} \sum_{k=1}^n \int_{a_k}^{b_k} h(s)^{-2}(-h'(s))  \left( \int_s^1f(se^{i\theta})^p h(t)t\nu(t)\,dt\right)\,ds \\
    &= (p')^{p-1} \lambda^{-p} \sum_{k=1}^n\Bigg(\left[ h(s)^{-1} \int_s^1f(se^{i\theta})^p h(t)t\nu(t)\,dt\, \right]_{a_k}^{b_k}\\
    &\quad-\int_{a_k}^{b_k} \left(-f(se^{i\theta})^p h(s)s\nu(s)\right) h(s)^{-1}\,ds\Bigg)\\
    &=(p')^{p-1} \lambda^{-p} \sum_{k=1}^n\Bigg(\left[h(s)^{-1}\int_s^1f(se^{i\theta})^p h(t)t\nu(t)\,dt\right]_{a_k}^{b_k}\\
    &\quad+\int_{a_k}^{b_k}f(se^{i\theta})^ps\nu(s)\,ds\Bigg).
    \end{split}
    \end{equation*}
By using that $p>1$ and the monotonicity of $h$, and by combining the above inequality with \eqref{eq:s1primaa}, it follows that
    \begin{equation*}\label{Eq:weak-ineq-main-estimate2}
    \begin{split}
   \frac{ \int_{E_\theta(\lambda)}s\eta(s)\,ds}{(p+\e)(p')^{p-1}\lambda^{-p} M_{p,\e}^p}
   &\le \sum_{k=1}^n \int_{a_k}^{b_k}f(se^{i\theta})^p s\nu(s)\,ds+
   \sum_{k=1}^{n-1} h^{-1}(b_k)\int_{b_k}^{a_{k+1}}f(se^{i\theta})^p h(s) s\nu(s)\,ds\\
   &\quad+ h^{-1}(b_n)\int_{b_n}^{1}f(se^{i\theta})^p h(s) s\nu(s)\,ds\\
   &\le\sum_{k=1}^n \int_{a_k}^{b_k}f(se^{i\theta})^p s\nu(s)\,ds
   +\sum_{k=1}^{n-1} \int_{b_k}^{a_{k+1}}f(se^{i\theta})^p  s\nu(s)\,ds\\
   &\quad+ \int_{b_n}^{1} f(se^{i\theta})^p  s\nu(s)\,ds\\
   &\le\int_{a_1}^{1}f(se^{i\theta})^p s\nu(s)\,ds.
    \end{split}
    \end{equation*}
By integrating this with respect to $\theta$, we deduce
    \begin{equation}\label{eq:final}
    \eta\left(\left\{z\in\D:T_\om(f)(z)\ge\lambda\right\}\right)\lesssim\lambda^{-p}M^p_{p,\ep}  \| f\|^p_{L^p_\nu},\quad \lambda>0,
    \end{equation}
for each non-negative compactly supported step function $f$. It remains to deduce this inequality for arbitrary $f\in L^p_\nu$. We will accomplish this in a couple steps, starting with simple functions.

Let $\lambda>0$, and let $f\neq 0$ be a non-negative simple function,   that is, of the form $f = \sum_{n=1}^N \alpha_n \chi_{E_n}$, where $\alpha_n \geq 0$ are non-negative constants and $E_n$ are disjoint measurable subsets of $\D$. Then, by the proof of \cite[Theorem~2.4 p. 71]{SteinShak} and \cite[Corollary~2.3 p. 71]{SteinShak},  there exists a sequence of non-negative (compactly supported) step functions $\varphi_k$ with their range contained in $\{\alpha_n : n=1,\ldots,N\}\cup\{0\}$ and converging to $f$ in $L^p_\nu$ and pointwise almost everywhere in~$\D$. Let $\d>0$ be arbitrary and denote $M=\max\{\alpha_n : n=1,\ldots,N\}$. Let $R\in(0,1)$ such that $\int_R^1 \eta(s)s\,ds < \frac{\d\lambda}{32M}$. By Egorov's theorem there exists a measurable set $E\subset\D$ such that $\om_1(\D\setminus E) < \frac{\d\lambda\widehat{\om}(R)}{16M\int_0^1 \eta(s)s\,ds}$ and $\varphi_k \to f$ uniformly on $E$, where $\om_1(z) = \frac{\om(z)}{|z|}$ for $z\in \D\setminus\{0\}$ is a finite measure on $\D$ because $\widehat{\om}(0) < \infty$. Then
    \begin{equation*}
    \begin{split}
    \widehat{\om}(z)|T_\om(\varphi_k)(z) - T_\om(f)(z)|
    &\leq \int_{|z|}^1\left|\varphi_k\left(s\frac{z}{|z|}\right) - f\left(s\frac{z}{|z|}\right)\right| \chi_E\left(s\frac{z}{|z|}\right) \om(s)\,ds \\
    &\quad+ 2M\int_{|z|}^1 \chi_{\D\setminus E}\left(s\frac{z}{|z|}\right) \om(s)\,ds
    \end{split}
    \end{equation*}
for all $z\in\D$. Since $\sup_{z\in\D} \left|\varphi_k(z) - f(z)\right| \chi_E(z) \to 0$ as $k\to\infty$,
for fixed $\lambda>0$
there exists $k_0\in\N$ such that $\sup_{z\in\D} \left|\varphi_k(z) - f(z)\right| \chi_E(z)<\frac{\lambda}{4}$ for all $k\ge k_0$.
So, for every point $z\in\D$ such that $\inf_{k\ge k_0}|T_\om(\varphi_k)(z) - T_\om(f)(z)| > \lambda/2$, it follows that
    \begin{equation*}
    \frac{2M}{\widehat{\om}(z)}\int_{|z|}^1 \chi_{\D\setminus E}\left(s\frac{z}{|z|}\right) \om(s)\,ds
    > \frac\lambda4.
    \end{equation*}
Thus
    \begin{equation*}
    \begin{split}
    \eta\left(\left\{z\in\D : |T_\om(\varphi_k)(z) - T_\om(f)(z)| > \lambda/2\right\}\right)
    &\leq \frac{8M}\lambda \int_\D \left(\int_{|z|}^1 \chi_{\D\setminus E}\left(s\frac{z}{|z|}\right) \om(s)\,ds \right)\frac{\eta(z)}{\widehat{\om}(z)}\,dA(z) \\
    &= \frac{8M}\lambda \int_0^1 \om_1((\D\setminus E)\setminus D(0,r))\,\frac{r\eta(r)}{\widehat{\om}(r)}\,dr
    \end{split}
    \end{equation*}
for $k\ge k_0$. Since
    \begin{equation*}
    \begin{split}
    \int_0^1 \om_1((\D\setminus E)\setminus D(0,r))\,\frac{r\eta(r)}{\widehat{\om}(r)}\,dr
    &\leq \frac{\om_1(\D\setminus E)}{\widehat{\om}(R)} \int_0^R\eta(r)r\,dr
    + 2 \int_R^1 \eta(r)r\,dr
    < \frac{\d\lambda}{8M},
    \end{split}
    \end{equation*}
we obtain $\eta\left(\left\{z\in\D : |T_\om(\varphi_k)(z) - T_\om(f)(z)| > \lambda/2\right\}\right) < \d$ for $k\ge k_0$.
Then take $\d=\frac{M^p_{p,\ep}\| f\|^p_{L^p_\nu}}{\lambda^p}$ and $k_1\ge k_0$ such that $\|\varphi_k\|_{L^p_\nu}\le 2 \| f\|_{L^p_\nu}$ for all $k \geq k_1$.
Then, bearing in mind \eqref{eq:final}, we deduce
    \begin{equation*}
    \begin{split}
    \eta(\{z\in\D : |&T_\om(f)(z)| \geq \lambda\}) \\
    &= \eta\left(\left\{z\in\D : |T_\om(f)(z)| \geq \lambda\right\}\cap\left\{z\in\D : |T_\om(\varphi_k)(z) - T_\om(f)(z)| \leq \lambda/2\right\}\right) \\
    &\quad+ \eta\left(\left\{z\in\D : |T_\om(f)(z)| \geq \lambda\right\}\cap\left\{z\in\D : |T_\om(\varphi_k)(z) - T_\om(f)(z)| > \lambda/2\right\}\right) \\
    &\leq \eta\left(\left\{z\in\D : |T_\om(\varphi_k)(z)| \geq \lambda/2\right\}\right) + \d \\
    &\lesssim M^p_{p,\ep}\lambda^{-p}\|\varphi_k\|^p_{L^p_\nu} + \frac{M^p_{p,\ep}\| f\|^p_{L^p_\nu}}{\lambda^p} \\
    &\lesssim M^p_{p,\ep}\lambda^{-p}\|f\|^p_{L^p_\nu}
    \end{split}
    \end{equation*}
for all $k\ge k_1$. Thus, we deduce that~\eqref{eq:final} holds for simple functions.

Finally, let $f\in L^p_\nu$ be arbitrary. Since $\left\{z\in\D:|T_\om(f)(z)|>\lambda \right\}\subset \left\{z\in\D:T_\om(|f|)(z)>\lambda \right\}$, we may assume that $f$ is non-negative. For each $N\in\N$, define $f_N = \max\{\min\{f,N\},\frac1N\}$. Then there exists a sequence of simple functions $\phi_k$ such that $N \geq \phi_k(z) \geq \phi_{k+1}(z) \geq f_N(z)$ for all $k\in\N$ and $z\in\D$ and $\phi_k \to f_N$ pointwise in $\D$ as $k\to\infty$. (Since $1/f_N$ is measurable, there exists an increasing sequence of simple functions $g_k$ such that $g_k(z)\geq\frac1N$ for all $z\in\D$, and $g_k\to f_N$ pointwise in $\D$ as $k\to\infty$. Take $\phi_k=1/g_k$. See \cite[Theorem 4.1 p. 31]{SteinShak}.) By the dominated convergence theorem there is $k_2$ such that $ \|\phi_k\|_{L^p_\nu}\le 2
 \|f_N\|_{L^p_\nu}$ for all $k\ge k_2$, which together with
\eqref{eq:final}  yields
    \begin{equation}\label{eq:weak-ineq-appr}
    \begin{split}
    \eta(\{z\in\D : T_\om(f_N)(z) > \lambda\})
    &\leq \eta(\{z\in\D : T_\om(\phi_k)(z) > \lambda\}) \\
    &\lesssim \lambda^{-p}M_{p,\e}^p \|\phi_k\|^p_{L^p_\nu} \\
    &\lesssim \lambda^{-p}M_{p,\e}^p \|f_N\|^p_{L^p_\nu} \\
    &\leq \lambda^{-p}M_{p,\e}^p \left(\|f\|^p_{L^p_\nu} + \frac{\nu(\D)}{N^p}\right), \quad \lambda>0, \quad N\in\N.
    \end{split}
    \end{equation}
Since $(\min\{f,N\})_{N\in\N}$ is an increasing sequence of non-negative measurable functions converging pointwise to $f$, $T_\om(\min\{f,N\})(z) \to T_\om(f)(z)$ for all $z\in\D$. Then, for fixed $\lambda>0$, by Egorov's theorem there is a measurable set $E\subset \D$ such that $\eta(E)<\lambda^{-p}M_{p,\e}^p \|f\|^p_{L^p_\nu}$ and $N_0$ such that $\sup_{z\in \D}|T_\om(f)(z)-T_\om(\min\{f,N\})(z)|\chi_{\D\setminus E}(z)<\frac{\lambda}{2}$ for all $N\ge N_0$. Hence
    \begin{equation*}
    \begin{split}
    &\eta(\{z\in\D : T_\om(f)(z) > \lambda\})
    \\ & = \eta(\{z\in \D\setminus E : T_\om(f)(z) > \lambda\})+ \eta(\{z\in E : T_\om(f)(z) > \lambda\})
   \\ &\le \eta(\{z\in\D : T_\om(\min\{f,N\})(z) > \lambda/2\})+ \lambda^{-p}M_{p,\e}^p \|f\|^p_{L^p_\nu}  \\
    &\leq \eta(\{z\in\D : T_\om(f_N)(z) > \lambda/2\}) + \lambda^{-p}M_{p,\e}^p \|f\|^p_{L^p_\nu}, \quad N\ge N_0,\quad\lambda>0.
    \end{split}
    \end{equation*}
Combining this with~\eqref{eq:weak-ineq-appr} and letting $N\to\infty$ yields~\eqref{eq:final} for $f\in L^p_\nu$, and thus (i) is satisfied. This finishes the proof of the theorem.
\end{Prf}
\medskip

\begin{Prf} \emph{Corollary~\ref{corollary:weird}.}
Assume that $T_\om:A^p_\nu\to L^p_\nu$ and $T_\om:L^p_\nu\to L^{p,\infty}_\nu$ are bounded. Then
    \begin{equation*}
    \DD_p(\om,\nu)=\sup_{0\le r<1}\frac{\widehat{\om}(r)^p}{\int_r^1 s\nu(s)\,ds}\int_0^r\frac{t\nu(t)}{\widehat{\om}(t)^p}\,dt<\infty
    \end{equation*}
and
    \begin{equation*}
    M_{p,\e}(\om,\nu,\nu)
    =\sup_{0\le r<1}\left(\widehat{\om}(r)^\e\int_0^r \frac{s\nu(s)}{\widehat{\om}(s)^{p+\ep}}\,ds\right)^\frac1p
    \left(\int_r^1\left(\frac{\om(s)}{s\nu(s)}\right)^{p'}s\nu(s)\,ds\right)^{\frac1{p'}}<\infty
    \end{equation*}
by Theorems~\ref{Theorem:T_w-bounded-intro} and \ref{Theorem:T_w-bounded-weak-L^p-case}. Moreover, \eqref{6new} implies
    $$
    \frac{\left(\int_r^1 s\nu(s)\,ds\right)^\frac1p}{\widehat{\om}(r)}
    \lesssim\left(\widehat{\om}(r)^\e\int_0^r\frac{s\nu(s)}{\widehat{\om}(s)^{p+\ep}}\,ds\right)^\frac1p.
    $$
Hence
    \begin{equation*}
    \begin{split}
    M_{p}(\om,\nu,\nu)
    &=\sup_{0\le r<1}\left(\int_0^r \frac{s\nu(s)}{\widehat{\om}(s)^{p}}\,ds\right)^\frac1p
    \left(\int_r^1\left(\frac{\om(s)}{s\nu(s)}\right)^{p'}s\nu(s)\,ds\right)^{\frac1{p'}}\\
   & \le \left(\DD_p(\om,\nu)\right)^\frac1p
    \sup_{0\le r<1}\frac{\left(\int_r^1 s\nu(s)\,ds\right)^{\frac{1}{p}}} {\widehat{\om}(r)}
    \left(\int_r^1\left(\frac{\om(s)}{s\nu(s)}\right)^{p'}s\nu(s)\,ds\right)^{\frac1{p'}}
   \\  & \lesssim \left(\DD_p(\om,\nu)\right)^\frac1p \sup_{0\le r<1}
    \left(\widehat{\om}(r)^\e\int_0^r\frac{s\nu(s)}{\widehat{\om}(s)^{p+\ep}}\,ds\right)^\frac1p
    \left(\int_r^1\left(\frac{\om(s)}{s\nu(s)}\right)^{p'}s\nu(s)\,ds\right)^{\frac1{p'}}
    \\ & \lesssim  \left(\DD_p(\om,\nu)\right)^\frac1p  M_{p,\e}(\om,\nu,\nu).
    \end{split}
    \end{equation*}
Thus $T_\om:L^p_\nu\to L^p_\nu$ is bounded. The converse implication is trivial.
\end{Prf}

\medskip

\begin{corollary}\label{counterexample}
Let $1<p<\infty$. For each  $\om\in\DD$ there exists $\nu\in\DD$ such that $T_\om:\, A^p_\nu\to L^p_\nu$ is bounded but $T_\om$ is not bounded from $L^p_\nu$ to $L^{p,\infty}_\nu$.
\end{corollary}

\begin{proof}
Fix $\om\in\DD$. Take $K>1$ and $r_n$ such that $\widehat{\om}(r_n)=K^{-n}\widehat{\om}(0)$.
Take  $t\nu(t)=\om(t) \sum_{n=0}^{\infty} \chi_{[r_{2n},r_{2n+1})}(t)$. Since $\om$ is not absolutely continuous with respect to $\nu$, $T_\om$ is not bounded from $L^p_\nu$ to $L^{p,\infty}_\nu$ by Theorem~\ref{Theorem:T_w-bounded-weak-L^p-case}.
\par On the other hand, since $\int_{t}^1 s\nu(s)\,ds\asymp \widehat{\om}(t)$, it is clear that $\nu\in\DD$ and
$$\int_0^r \frac{t\nu(t)}{\widehat{\om}^p(t)}\,dt\le   \int_0^r \frac{ \om(t)dt}{\widehat{\om}(t)^p}\lesssim
 \frac{1}{\widehat{\om}(r)^{p-1}}\asymp
\frac{\int_r^1 t\nu(t)\,dt}{\widehat{\om}(r)^p},\quad 0\le r<1,$$
Thus, $T_\om \,: A^p_\nu\to L^p_\nu$ is bounded by Theorem~\ref{Theorem:T_w-bounded-intro}.
\end{proof}

We finish the paper providing three weights,
$\om(s)=s$, $\nu(s)=(1-s)^{p-1}\left(\log\left( \frac{e}{1-s}\right)\right)^{2(p-1)}$
and $\eta(s)=(1-s)^{p-1}\left(\log\left(\frac{e}{1-s}\right)\right)^{
p-1}$ such that $T_\om:L^p_\nu\to L^{p,\infty}_\eta$ is bounded for each $1<p<\infty$ by Theorem~\ref{Theorem:T_w-bounded-weak-L^p-case},
but $T_\om:L^p_\nu\to L^p_\eta$ is unbounded by Theorem~\ref{Theorem:T_w-bounded-L^p-case}.

\medskip

{\em{Acknowledgements.}}
 We would like to thank professor Francisco J. Mart\'{\i}n-Reyes for helpful conversations about
Hardy operators and pointing out relevant references on the topic.

\end{document}